\documentclass[a4paper,10pt]{article}
\usepackage[utf8]{inputenc}
\usepackage{amsmath,amsthm}
\usepackage{amsfonts}
\usepackage{amssymb}
\usepackage{a4wide}
\usepackage{todonotes}
\usepackage{amscd}
\usepackage{tikz}
\usepackage{diagbox}

\newcommand{\R}{\mathbb{R}}
\newcommand{\N}{\mathbb{N}}

\newcommand{\Ha}{H}   
\newcommand{\Ti}{A} 
\newcommand{\tTi}{\tilde{\Ti}}
\newcommand{\Tii}{B}

\newcommand{\Talt}{\tilde{T}} 

\newcommand{\X}{\mathcal{X}}
\newcommand{\Y}{\mathcal{Y}}
\newcommand{\Z}{\mathcal{Z}}

\newcommand{\W}{\mathcal{W}}

\newcommand{\range}{\mathcal{R}}

\newtheorem{thm}{Theorem}

\newtheorem{lem}{Lemma}

\newtheorem{Def}{Definition}

\newtheorem{pro}{Proposition}
\newtheorem{rem}{Remark}
\newtheorem{cor}{Corollary}
\newcommand{\fs}{\footnotesize}
\newcommand{\bb}{\large \bf }

\title{Curious ill-posedness phenomena in the composition of non-compact linear operators in Hilbert spaces}
\author{Stefan Kindermann\footnotemark[2]
\and Bernd Hofmann\footnotemark[3]}

\begin{document}

\maketitle\footnotetext[2]{Industrial Mathematics Institute, Johannes Kepler University Linz, Alternbergergstraße 69, 4040 Linz, Austria. Email: kindermann@indmath.uni-linz.ac.at}
\footnotetext[3]{Chemnitz University of Technology, Faculty of Mathematics, 09107 Chemnitz, Germany.\\ Email: hofmannb@mathematik.tu-chemnitz.de}

\begin{abstract}
We consider the composition of
operators with non-closed range in
Hilbert spaces and how the nature of
ill-posedness is affected by their composition. Specifically, we study the
\mbox{Hausdorff-,} Ces\`{a}ro-,  integration
operator, and their adjoints, as well as some combinations of those. For the composition
of the Hausdorff- and the
Ces\`{a}ro-operator,
we give
estimates of  the decay of the corresponding singular values.
As a curiosity, this provides also
an example of two practically relevant
non-compact operators, for which
their composition is compact.
Furthermore, we characterize those operators for which a composition with 
a non-compact operator gives a compact one.
\end{abstract}

\bigskip

{\parindent0em {\bf MSC2020:}}
47A52, 65J20, 40G05, 44A60

\bigskip

{\parindent0em {\bf Keywords:}}
Linear ill-posed problems, non-compact and compact operators, composition operators, degree of ill-posedness, Hausdorff moment operator, Ces\`{a}ro operator, integration operator, singular value decomposition.

\bigskip

\section{Introduction}

Let $\X,\Y,$ and $\Z$ denote infinite-dimensional real separable Hilbert spaces. In this note, we consider composite operators $T$ that are factorized as
 \begin{equation*}
 \begin{CD}
  T:\; @.  \X @> \Tii >> \Z  @> \Ti>> \Y\,,
  \end{CD}
\end{equation*}
where $\Tii: \X \to \Z$, $\Ti: \Z \to \Y$, and consequently, 
$T=\Ti\, \Tii: \X \to \Y$ are bounded linear operators with non-closed range, which means that zero belongs to the spectrum of the operators $\Ti$, $\Tii$, and $T$.
So the composite equation
\begin{equation}\label{eq:opeq1}
T\,x\,=\,\Ti\,(\Tii\,x)\,=\,y\qquad (x \in \X,\;y\in \Y)\,,
\end{equation}
but also the outer equation
\begin{equation}\label{eq:opeq2}
\Ti\,z\,=\,y \qquad (z \in \Z,\;y\in \Y)
\end{equation}
and the inner equation
\begin{equation}\label{eq:opeq3}
\Tii\,x\,=\,z \qquad (x \in \X,\;z\in \Z)
\end{equation}
represent ill-posed linear operator equations and can serve as models for inverse problems characterized by forward operators $\Ti$, $\Tii$, and $T$ with non-closed dense ranges $\range(T) \subset \Y$, $\range(\Tii) \subset \Z$, and  $\range(\Ti) \subset \Y$, respectively. The stated assumptions furthermore imply that the corresponding adjoint operators $T^*, \Ti^*$, and $\Tii^*$ are also bounded linear operators. Note that when only \eqref{eq:opeq2} is considered together with the assumption $z = \Tii x$ and $\Tii$ 
being compact,
then this problem is not ill-posed as the solution is searched in the compact set $\range(\Tii)$. This follows 
from the well-known Tikhonov theorem on the inverses on compact sets.

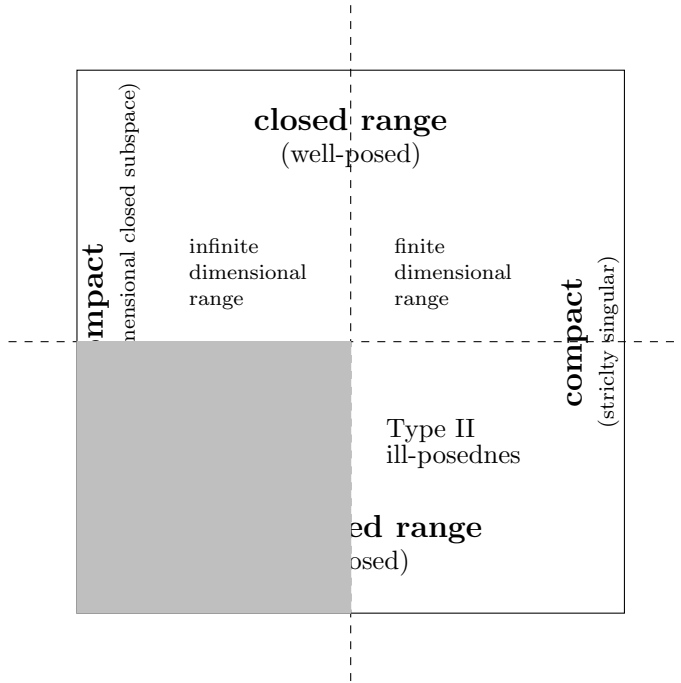
\begin{figure}
\begin{center}
\begin{tikzpicture}[scale=0.9]
\draw[draw=black] (0,0) rectangle ++(8,8);
\draw[dashed] (-1,4) -- (9,4);
\draw[dashed] (4,-1) -- (4,9);
\node[align=left] at (2.5,5.0) {\fs infinite\\[-0.5ex] \fs  dimensional\\[-0.5ex]  \fs  range}; 
\node[align=left] at (5.5,5.0) {\fs finite\\[-0.5ex] \fs  dimensional\\[-0.5ex]  \fs  range}; 
\node[align=left]  at (5.5,2.5) { Type II \\[-0.5ex]  ill-posednes};
\node[align=left] at (2.5,2.5) { Type I \\[-0.5ex]  ill-posednes};
\node[align=center] at (4.0,7) {\bb closed range\\ (well-posed)};
\node[align=center] at (4.0,1) {\bb non-closed range\\ (ill-posed)};
\node[align=center,rotate=90] at (0.5,4) {\bb non-compact\\ \fs (range contains infinite-dimensional closed subspace)};
\node[align=center,rotate=90] at (7.5,4) {\bb compact\\ \fs (striclty singular)};
\fill[draw=black,color=lightgray,opacity=0.2] (0,0) rectangle (4,4);
\end{tikzpicture} 
\end{center}
\caption{Case distinction for bounded linear operators mapping between infinite-dimensional Hilbert spaces.}\label{figcase}
\end{figure}
In this context we recall the paper \cite{Nashed86}, where Nashed distinguished for such operator equations {\sl ill-posedness of type~I} when the forward operator is non-compact and, as alternative, {\sl ill-posedness of type~II} when the forward operator is compact.
See Figure~\ref{figcase} for an illustration of the concepts in Hilbert spaces. The operators of special interest in this article are from the gray-shaded part.

Unfortunately, only for type~II the {\sl strength and degree of ill-posedness} caused by the forward operator can be simply expressed by the decay rate of the associated singular values of this operator; see Definitions~\ref{def:degree1} and \ref{def:degree2} below.
For discussions about the degree of ill-posedness of equations \eqref{eq:opeq1} with non-compact operators that are ill-posed of type~I in the sense of Nashed, we refer to the articles \cite{HofFlei99,HofKind10,nmh22}. Here, however, we assume in the sequel that the composite operator $T:\X \to \Y$ in \eqref{eq:opeq1}
is compact and possesses the singular system $$\{\sigma_i(T)>0,u_i \in \X ,v_i \in \Y\}_{i=1}^\infty\,,$$ with decreasingly ordered singular values
$$\|T\|_{\scriptscriptstyle \mathcal{L}(\X,\Y)}=\sigma_1(T) \ge \sigma_2(T) \ge \ldots \ge \sigma_i(T) \ge \sigma_{i+1}(T) \ge \ldots$$ tending to zero as $i \to \infty$ and complete orthonormal systems
$\{u_i\}_{i=1}^\infty$ in $\X$ and $\{v_i\}_{i=1}^\infty$ in $\Y$ obeying $Tu_i=\sigma_i(T)\,v_i$ as well as $T^*v_i=\sigma_i(T)\,u_i$, for all $i \in \N$.
\begin{Def}[Mild, moderate, and severe ill-posedness] \label{def:degree1}
Let the bounded linear operator $T: \X \to \Y$ be compact. Then we call the operator equation \eqref{eq:opeq1}
\begin{itemize}
   \item \textsl{mildly ill-posed} whenever the decay rate of $\sigma_i(T) \to 0$, as $i \to \infty$, is slower than any polynomial rate.
	\item \textsl{moderately ill-posed} whenever the decay rate of $\sigma_i(T)\to 0$, as $i \to \infty$, is polynomial.
	\item \textsl{severely ill-posed} whenever the decay rate of $\sigma_i(T)\to 0$, as $i \to \infty$, is higher than any polynomial rate.
\end{itemize}
\end{Def}

{\parindent0em Along} the lines of \cite{HofTau97} (see also \cite{HofKind10}), one can define in more detail an interval and a degree of ill-posedness as follows.

\begin{Def}[Interval and degree of ill-posedness] \label{def:degree2}
We denote, for an ill-posed operator equation~\eqref{eq:opeq1} with a compact forward operator $T$, the well-defined interval of the form
\begin{equation} \label{eq:interval}
[\underline{\kappa},\overline{\kappa}]=\left[\liminf \limits _{i \to \infty}
\frac{-\log(\sigma_i(T))} {\log(i)}\,,\,\limsup \limits _{i \to
\infty} \frac{-\log(\sigma_i(T))} {\log(i)}\right]  \subset [0,\infty]
\end{equation}
as {\sl interval of ill-posedness}. If $\underline{\kappa}$ and $\overline{\kappa}$ from $[0,\infty]$ are both finite positive, then we have {\sl moderate ill-posedness}, and if they even coincide as $\underline{\kappa}=\overline{\kappa}=\kappa$, then we call the equation {\sl ill-posed of degree} $\kappa>0$.  {\sl Severe ill-posedness} occurs if
the interval degenerates as $\underline{\kappa}=\overline{\kappa}=\infty$, and vice versa, {\sl mild ill-posedness} is characterized by a degeneration as  $\underline{\kappa}=\overline{\kappa}=0$.
\end{Def}

\section{A selection of linear operators with non-closed range}

To investigate the ill-posedness behaviour of composite operators $T=\Ti\, \Tii$ in the equation \eqref{eq:opeq1}, we present a selection of bounded compact and non-compact operators with non-closed range that can be exploited for $\Ti$ and $\Tii$. We start with the {\sl simple integration operator} $J: L^2(0,1) \to L^2(0,1)$ and its adjoint operator $J^*: L^2(0,1) \to L^2(0,1)$ defined as
\begin{equation} \label{eq:J}
[J\,x](s):=\int _0^s x(t) \,dt \qquad (0 \le s \le 1, \quad x \in L^2(0,1))\,,
\end{equation}
and
\begin{equation} \label{eq:Jstar}
[J^*x](t):=\int _t^1 x(s)\, ds \qquad (0 \le t \le 1, \quad x \in L^2(0,1))\,,
\end{equation}
respectively. Both operators are {\sl compact} and so is the self-adjoint specific {\sl diagonal operator} $D: \ell^2 \to \ell^2$, which appears here as
\begin{equation} \label{eq:D}
[D\,y]_j:=\frac{y_j}{j}  \qquad (j=1,2,\ldots,\quad y \in \ell^2).
\end{equation}
It is well-known for $J$ and $J^*$ and evident for $D$ that the degree of ill-posedness is {\sl one}. The singular system of $J$ and $J^*$ can be written down in an explicit manner, where we have $\sigma_i(J) \asymp i^{-1}$, for $i \to \infty$, as singular value asymptotics.
The singular system of $D$ is of the form $\{i^{-1},e^{(i)},e^{(i)}\}_{i=1}^\infty$, where $e^{(i)}$ denotes the $i$-th unit vector in $\ell^2$.

The {\sl Ces\`{a}ro operator} $C:L^2(0,1) \to L^2(0,1)$ and its adjoint operator $C^*: L^2(0,1) \to L^2(0,1)$, which attain the form
\begin{equation} \label{eq:C}
[C\,x](s):= \frac{1}{s}\,\int_0^s x(t)\,dt \qquad (0 \le s \le 1, \quad x \in L^2(0,1))\,
\end{equation}
and
\begin{equation} \label{eq:Cstar}
 [C^*x](t):= \int_t^1 \frac{x(s)}{s}\,\,ds \qquad (0 \le t \le 1, \quad x \in L^2(0,1))\, ,
\end{equation}
respectively, are non-compact operators with non-closed range; see \cite{Brown65,Leib73}. A further interesting non-compact operator with non-closed range connecting the spaces $L^2(0,1)$ and $\ell^2$ is the {\sl Hausdorff moment operator} $\Ha: L^2(0,1) \to \ell^2$ defined as
\begin{equation} \label{eq:A}
[\Ha\,x]_j:=\int _0^1 x(t)\, t^{j-1} dt \qquad (j=1,2,\ldots, \quad x \in L^2(0,1))\,,
\end{equation}
with the corresponding adjoint operator $\Ha^*:\ell^2 \to L^2(0,1)$ of the form
\begin{equation} \label{eq:Astar}
[\Ha^*y](t):= \sum _{j=1}^\infty y_j\, t^{j-1} \qquad (0 \le t \le 1,\quad y \in \ell^2),
\end{equation}
and we refer for details to \cite{Gerth21} (see also \cite{Gerth22,Haus23,HofMat22}).

For the last four non-compact operators, a degree or interval of ill-posedness in the sense of Definition~\ref{def:degree2} does not make sense. But if those operators occur as $\Ti$ or $\Tii$ in a composition $T=\Ti \, \Tii$, where $T$ is compact, then they can substantially influence
the degree of ill-posedness for $T$. This is also the case for  non-compact {\sl multiplication operators} $M: L^2(0,1) \to L^2(0,1)$ with non-closed range
\begin{equation} \label{eq:M}
[M\,x](t):=m(t)\,x(t) \qquad (0 \le t \le 1, \quad x \in L^2(0,1))\,,
\end{equation}
for which the multiplier functions $m \in L^\infty(0,1)$ possess essential zeros in $(0,1)$. We refer in this context also to the papers \cite{Hof06,HofWolf05}.

\section{Can a non-compact operator in composition destroy the degree of ill-posedness of a compact operator?}

In the past years, equations \eqref{eq:opeq1} with compact composite operators $T=\Ti \, \Tii$ have been studied under the assumption that $\Tii$ is compact and $\Ti$ is a non-compact operator with non-closed range. It had been an open question whether the non-compact operator $\Ti$ can
amend the degree of ill-posedness of the compact operator $\Tii$ in the composition $T$.

The first studies in \cite{Freitag05,HofWolf05,HofWolf09} investigated the case $T=M \, J: L^2(0,1) \to L^2(0,1)$ with multiplication operators $\Ti:=M:L^2(0,1) \to L^2(0,1)$ from \eqref{eq:M} and the simple integration operator $\Tii:=J:L^2(0,1) \to L^2(0,1)$ from \eqref{eq:J}.
Indeed, all these studies indicated the asymptotics $\sigma_i(T) \asymp i^{-1}$ as $i \to \infty$, even for multiplier functions $m$ with strong (exponential-type) zeros that occur in inverse problems of option pricing; see~\cite{HeinHof03}.
This means that along the lines of those studies, the non-compact multiplication operator $M$ does not destroy the degree of ill-posedness {\sl one} of the compact operator $J$ in such composition.

However, the situation changed when for $\Tii:=J$,  the multiplication operator $M$   was replaced with the non-compact Hausdorff moment operator $\Ti:=\Ha:L^2(0,1) \to \ell^2$ from \eqref{eq:A}. In the article \cite{HofMat22}, the assertion of the following proposition could be shown in the context of Corollary~2 and Theorem~3 ibid.

\begin{pro} \label{pro:HM}
The operator $T=\Ha \, J: L^2(0,1) \to \ell^2$ with $\Ha:L^2(0,1) \to \ell^2$ from \eqref{eq:A} and $J:L^2(0,1) \to L^2(0,1)$ from \eqref{eq:J} obeys for some positive constants $\underline{c}$ and $\overline{c}$ the inequalities
\begin{equation} \label{eq:pro1}
\exp(-\underline{c}\,i) \le \sigma_i(T) \le \frac{\overline{c}}{i^{3/2}}
\end{equation}
for sufficiently large indices $i \in \mathbb{N}$.
\end{pro}
As a consequence of Proposition~\ref{pro:HM}, the interval of ill-posedness for the composition $\Ha \, J$ is a subset of the interval $[\frac{3}{2},\infty]$. This was the first example in the literature to demonstrate with respect to $\Tii:=J$ the degree-destroying potential of a non-compact operator $\Ti$ in such a composition. Unfortunately, by now it could not be cleared if $T=\Ha \, J$ really leads to an exponentially (severely) ill-posed problem or whether it leads to a moderate ill-posed problem.

So it was exciting to replace the Hausdorff moment operator $\Ha$ (used as $\Ti$) with the non-compact  Ces\`{a}ro operator
$\Ti:=C: L^2(0,1) \to L^2(0,1)$ from \eqref{eq:C} in the composition $T=C \, J$. The recent paper~\cite{DFH24} has proven that we have, for such $T$, the asymptotics $\sigma_i(T) \asymp i^{-2}$ as $i \to \infty$, which means that the degree of ill-posedness is {\it two} for $T=C \, J$.
Hence, $C$ increases in that composition the degree of ill-posedness of $J$ just by one. Taking into account that $J^2=M \, T$ with the multiplication operator $M:L^2(0,1) \to L^2(0,1)$ and the multiplier function $m(t)=t$, one can see again that such multiplication operator does not
amend the degree of ill-posedness, because the asymptotics $\sigma_i(J^2) \asymp i^{-2}$, as $i \to \infty$, is well-known; see for example \cite{Ramlau20}.

\section{The curious case that the composition of two non-compact operators is compact}

It was surprising for the authors that also two
{\em non-compact} operators $\Ti$ and $\Tii$ with non-closed range can generate a {\em compact} operator by composition,
$T=\Ti \, \Tii$.
Indeed, let $\Tii:=C^*: L^2(0,1) \to L^2(0,1)$ from \eqref{eq:Cstar} and $\Ti:=\Ha: L^2(0,1) \to \ell^2$ from \eqref{eq:A}.
Then we have such a situation as the next proposition indicates.

\begin{pro} \label{pro:fact}
The operator $T: L^2(0,1) \to \ell^2$ defined as $T:=\Ha \, C^*$ with the non-compact operators $\Ha$ from \eqref{eq:A} and $C^*$ from \eqref{eq:Cstar} is compact and even a Hilbert-Schmidt operator.
\end{pro}
\begin{proof}
We have that
$$T=\Ha \, C^*=D \, \Ha\,,$$
with the {\sl compact} diagonal operator $D: \ell^2 \to \ell^2$ from \eqref{eq:D}. This can be seen by inspection of the $j$-th component of $Tx$, which can be written as
$$[\Ha(C^*x)]_j=\int_0^1 \left(\int_t^1 \frac{x(s)}{s}\,ds \right)\,t^{j-1}\,dt\,. $$
Integration by parts yields moreover
$$[\Ha(C^*x)]_j= \frac{1}{j} \int_0^1 x(t)\,t^{j-1}\,dt =\frac{1}{j}\,[\Ha x]_j=[D(\Ha x)]_j\,. $$
Since $D$ is a compact operator, this property carries over to the composition $T=D \, \Ha$ of $D$ with the bounded linear operator $\Ha$. In the same manner, the Hilbert-Schmidt operator $D$ with the Hilbert-Schmidt norm $\|D\|_{HS}=\sqrt{\sum_{i=1}^\infty \frac{1}{i^2}}<\infty$ leads to a Hilbert-Schmidt property of $T$ by favour of the inequality $\|T\|_{HS} \le \|D\|_{HS}\,\|\Ha\|_{\scriptscriptstyle \mathcal{L}(L^2(0,1),\ell^2)}$.
\end{proof}
\begin{rem} {\rm We note that of course the same fact can also be formulated for the adjoint operator $T^*=C \,  \Ha^*:\ell^2 \to L^2(0,1)$, where $C$ from \eqref{eq:C} and $\Ha^*$ from \eqref{eq:Astar} are again non-compact operators with non-closed range, but $T^*=\Ha^*\, D$ is compact. Let us recall that  the adjoint of an operator completely mirrors the 
ill-posedness of the original operator; see, for instance, \cite{Gerth21}: $T$ is ill-posed of type I (resp. II) if and 
only if $T^*$ is ill-posed of type I (resp. II). Moreover, it follows by the properties of the singular values that the degree of ill-posedness of  $T$ is identical to that of $T^*$. On the other hand, calculating the 
singular values seems to be equally difficult for both operators, $T$, $T^*$.
\hfill\fbox{}}
\end{rem}

\begin{rem}
\rm Although the phenomenon $\text{Non-compact} \circ \text{Non-compact} = \text{Compact}$ is 
well-known in the literature for ``academic examples''---see the Section~\ref{sec:ncnc} below---we 
found the above situation  ``curious'', since 
it occurs for two operators (Ces\`{a}ro- and Hausdorff operator) 
arising from practically relevant problems. We will show  in Section~\ref{sec:ncnc} that this is actually 
a characteristic of type I ill-posedness.
\end{rem}

Along the lines of the proof of \cite[Theorem~3]{HofMat22} one can prove the following theorem.

\begin{thm} \label{thm:main}
For the composition $T=D \, \Ha=\Ha \, C^*: L^2(0,1) \to \ell^2$ with the operators\linebreak \mbox{$\Ha: L^2(0,1) \to \ell^2$} from \eqref{eq:A}, $D: \ell^2 \to \ell^2$ from \eqref{eq:D} and $C^*: L^2(0,1) \to L^2(0,1)$ from \eqref{eq:Cstar},
there exists a positive constant $c$ such that
\begin{equation} \label{eq:rough2}
\sigma_i(T)\le \frac{c}{i^{3/2}}\quad (i=1,2,\ldots),
\end{equation}
and the degree of ill-posedness of $T$ is at least 3/2.
\end{thm}
\begin{proof}
A main tool for the proof is the system $\{L_j\}_{j=1}^\infty$ of shifted Legendre polynomials
which represent a complete orthonormal system in the Hilbert space $L^2(0,1)$.
This system is the result of the Gram-Schmidt orthonormalization process of the system $\{t^{j-1}\}_{j=1}^\infty$ of
monomials. Consequently, we have
$${\operatorname{span}}(1,t,\ldots,t^{j-1})={\operatorname{span}}(L_1,L_2,\ldots,L_j).$$
Hence, we have for $m \ge 2$ that $L_m \perp t^{j-1}$,
for all $1 \le j < m$.
As has been proven by \cite[Proposition~3]{HofMat22}, we have for the Hilbert-Schmidt operator $T=D \, \Ha$
$$\sum_{i=n+1}^\infty \sigma_i^2(T) \le \|T(I-Q_n)\|_{HS}^2\,, $$
where $Q_n$ denotes the projection onto $\operatorname{span}(L_1,...,L_n)$.
From that we derive here the estimates
\begin{equation} \label{eq:sq}
\sum_{i=n+1}^\infty \sigma_i^2(T) \le \sum_{i=n+1}^\infty \|T(I-Q_n)L_i\|_{\ell^2}^2 = \sum_{i=n+1}^\infty \|T\,L_i\|_{\ell^2}^2=\sum_{i=n+1}^\infty \sum_{j=1}^\infty \langle D(\Ha L_i),e^{(j)} \rangle_{\ell^2}^2\,.
\end{equation}
By exploiting the system of normalized functions
$$ h_{j}(s):= \sqrt{2j+1}\,s^{j}\,\in L^{2}(0,1) \quad (j=0,1,2,\dots),$$
we can rewrite the terms of the form $\langle D(\Ha L_i),e^{(j)} \rangle_{\ell^2}^2$ in \eqref{eq:sq} as
$$ \langle D(\Ha L_i),e^{(j)} \rangle_{\ell^2} = \frac{1}{j^2}\langle \Ha  L_i,e^{(j)} \rangle_{\ell^2}= \frac{1}{j^2}\left(\int_0^1\frac{h_{j-1}(s)\,L_i(s)\,ds}{\sqrt{2j-1}} \right)^2=\frac{1}{j^2(2j-1)}\langle h_{j-1},L_i \rangle_{L^2(0,1)}^2\,.   $$
Taking into account $\|h_j\|_{L^2(0,1)}=1$  and the orthogonality relations between $h_j$ and $L_i$ we derive now from \eqref{eq:sq} the estimate
$$\sum_{i=n+1}^\infty \sigma_i^2(T) \le \sum_{j=n+2}^\infty \frac{1}{j^2(2j-1)}\sum_{i=n+1}^\infty \langle h_{j-1},L_i \rangle_{L^2(0,1)}^2=\sum_{j=n+2}^\infty \frac{1}{j^2(2j-1)}\|(I-Q_n)h_{j-1}\|_{L^2(0,1)}^2\,,  $$
and with $\|(I-Q_n)h_{j-1}\|_{L^2(0,1)} \le 1$, we can further estimate as
$$\sum_{i=n+1}^\infty \sigma_i^2(T) \le \sum_{j=n+2}^\infty \frac{1}{j^2(2j-1)} \le c_1\,n^{-2}\,   $$
for some constant $c_1>0$. We recall now from \cite[Lemma~4]{HofMat22} the fact that an estimate  $$\sum_{i=n+1}^\infty \sigma_i^2(T) \le c_1\,n^{-2\gamma}\;\;(n \in \N), \;\; \mbox{for} \;\; \gamma>0\;\;\mbox{and} \;\; c_1>0\,,$$  implies
the existence of a constant $c_2>0$ such that $\sigma_i^2(T) \le c_2\,i^{-(2\gamma+1)}\;(i \in \N)$. Applying this fact with $\gamma=1$ yields the inequality \eqref{eq:rough2}, which completes the proof.
\end{proof}
\begin{rem} {\rm The singular system of the compact operator $D: \ell^2 \to \ell^2$ mentioned above indicates the asymptotics $\sigma_i(D) \asymp i^{-1}$ as $i \to \infty$. From Theorem~\ref{thm:main}, however, we see that there is some constant $K>0$ such that
$$
\sigma_{i}(D \, \Ha )/\sigma_{i}(D) \leq
\frac{K}{i^{1/2}}\quad (i=1,2,\ldots).$$
Consequently, as for the composition $\Ha  \, J$ (see Proposition~\ref{pro:HM}) also here for $D \, \Ha $ the non-compact operator $\Ha $ has the power to increase the decay rate of the singular values of the respective compact operators by an exponent of at least $1/2$.
\hfill\fbox{}
}\end{rem}

As in Proposition~\ref{pro:HM} (cf.~\cite[Corollary~3.6]{HofMat22}) for the composition $\Ha  \, J $, one can also here verify for the composition
$D \, \Ha = \Ha \, C^*$ lower bounds  of exponential type for the singular values. We make this explicit by the following theorem.

\begin{thm} \label{thm:lowerbounds}
Consider the operator $T  =  \Ha  \, C^*=D \, \Ha $ from Theorem~\ref{thm:main}.
Then we have the lower bound
\begin{equation} \label{eq:lowb}
\frac{c_0}{i} \exp(-2 i) \leq  \sigma_i(T) \quad (i=1,2,\ldots)\,,
\end{equation}
with some constant $c_0>0$.
\end{thm}
\begin{proof}
Consider the operator $T T^* :\ell^2 \to \ell^2$. Since $\Ha  \Ha ^*$ is the infinite Hilbert matrix $\mathcal{H}$, we observe that
\[ T T^*  = D \mathcal{H} D\]
with $D$ the diagonal operator $D = {\rm diag}((\tfrac{1}{i})_i)$ from \eqref{eq:D}.
Let $P_N: \ell^2 \to \ell^2$ be the projection onto the
first $N$ components and let $\mathcal{H}_N = P_N \mathcal{H} P_N$ be the first $n\times n$-segment of the Hilbert matrix.
Then the estimate
\[ \sigma_N(P_N T T^* P_N) \leq
\|P_N\|^2 \sigma_N( T T^* ) =  \sigma_N( T T^* ) \]
holds.
On the other hand,  $P_N$ commutes with $D$. Now let $D_N = P_N D P_N $ be the $N\times N$-segment
of $D$, which means that $D_n = {\rm diag}((1/i)_{i=1,N})$.
Under such setting we consequently have
\[ P_N T T^* P_N =
D_N \mathcal{H}_N D_N. \]
It is well-known from \cite{Tod54} that there is a constant $c>0$ in the context of an estimate from above
for the norm of the inverse of the finite  Hilbert matrix $\mathcal{H}_N$ as
\[\|\mathcal{H}_N^{-1}\| \leq c \exp(4 N). \]
This gives
\begin{align*}
 \sigma_N(P_N T T^* P_N) &=
 \frac{1}{\|(D_N \mathcal{H}_N D_N)^{-1} \|_{\R^N\to \R^N}} =
 \frac{1}{\|D_N^{-1} \mathcal{H}_N^{-1}  D_N^{-1} \|_{\R^N\to \R^N} } \\
&\geq \frac{1}{ \|D_N^{-1}\|^2
\|  \mathcal{H}_N^{-1} \|_{\R^N\to \R^N} }
\geq \frac{1}{ N^2 c \exp( 4 N) }
\end{align*}
and yields the claimed result \eqref{eq:lowb}
by taking into account that $\sigma_N(T)^2 = \sigma_N(T T^*) $.
\end{proof}

\begin{rem}
{\rm
Table~\ref{tab1} gives an overview of
known estimates for the  singular values of
the composition of certain operators.
By inspecting the estimates \eqref{eq:pro1} as well as \eqref{eq:rough2} and \eqref{eq:lowb}, it is a really challenging question whether the compositions $\Ha  \, J$ and $D \, \Ha $ may lead to moderately ill-posed problems, although the character of the Hausdorff moment operator $\Ha $ seems to be severely ill-posed as the paper \cite{Rosenblum58} indicates.
If the answer is {\it yes}, then the moderate decay of the singular values of $J$ and $D$ has the power to stop in such compositions the severe ill-posedness character of $\Ha $ expressed by an exponential decay of the corresponding multiplier
function in the spectral decomposition of $\Ha   \Ha ^*$ (infinite Hilbert matrix).
\hfill\fbox{}}
\end{rem}

\begin{table}
\begin{center}
\begin{tabular}{ c||c|c|c }
 \diagbox{$\Tii$}{$\Ti$} & $\Ha $ & $C$ & $M$   \\ \hline
\rule{0mm}{2.7ex}$J$ & $e^{-c i} \lesssim \sigma_i \lesssim  i^{-\frac{3}{2}}$ &
$ \sigma_i \sim i^{-2}$  &
$ \sigma_i \sim i^{-1}$  \\  \hline
\rule{0mm}{2.7ex}$C^*$ &
$ i^{-1} e^{-2 i} \lesssim  \sigma_i \lesssim   i^{-\frac{3}{2}}$
 & 
\end{tabular}
\caption{Overview of known bounds for the singular values of compositions
of certain operators.}\label{tab1}
\end{center}
\end{table}

By the above example of  compactness of the composition
two non-compact operators, the following
issue is raised
that seems to be trivial only at first glimpse:
When $T$ is a non-compact operator
between Hilbert spaces, is the selfadjoint operator
$T^*T$ also non-compact? Clearly, $T$ is non-compact if and
only if $T^*$ is (by Schauder's theorem). However, as
we have seen, this does not necessarily imply non-compactness of
the composition. Using polar decomposition, the
following lemma can be shown:
\begin{lem}
Let $T: \X \to \Y$ be a bounded linear operator between
Hilbert spaces $\X,\Y$. Then
\begin{equation}
T \text{ is compact}
\Leftrightarrow
T^* \text{ is compact}
\Leftrightarrow
T^*T \text{ is compact}.
 \end{equation}
\end{lem}
\begin{proof}
As mentioned, the first equivalence is Schauder's theorem
(see, e.g. \cite[Thm 4.19]{Rud91}),
and since compact operators form an ideal, we only
have to show that if $T^*T$ is compact, then
$T^*$ is compact. Compactness of $T^*T$
implies compactness of the square root $\sqrt{T^*T}$, as can be shown by a spectral decomposition. Now  the
polar decomposition $T^* = \sqrt{T^*T} U $,
e.g., \cite[Thm 12.35]{Rud91} or \cite[Thm~3.9]{Conway00},
with a bounded (unitary) operator  $U$ implies compactness
of $T^*$.
\end{proof}

\section{Numerical illustration of the decay of the singular values} \label{sec:numerics}
In order to illustrate the findings of Theorems~\ref{thm:main} and \ref{thm:lowerbounds},
we provide some numerical calculations of the singular values of $T = \Ha \, C^* = D \, \Ha $, 
and, for comparison, also for the operator in Proposition~\ref{pro:HM}, 
namely $\Talt = \Ha  \, J$. Both involve the Hausdorff operator and are composed with 
either the Ces\`{a}ro or the integration operator. 
The results for $\Talt$ are for comparison only as similar numerics
was performed in \cite{Gerthnum}. 
Let us spoil the story by noting that the conclusions from these 
numerical calculations do not give evidence of a certain decay of the 
singular values for the original operators $T$, $\Talt$ because of the 
exponential ill-posedness of the discretized operators and the corresponding 
huge margins of uncertainty left.

We discretized $T$ and $\Talt$ each in two ways:  

1.) The first one is obtained  by discretizing  from the left by using $P_N:\ell^2 \to \ell^2$ 
the projection onto the first $N$ components. Then, we numerically calculate 
the eigenvalues of $(P_N T)(P_N T)^* = P_N T T^* P_N$ and similar for $\Talt$. 
The resulting finite-dimensional matrix in both cases involves  discrete Hilbert-type matrices, 
for which an exponential decay of the singular values is known;
cf.~$\mathcal{H}_N$  and the statements in the proof of Theorem~\ref{thm:lowerbounds}. 
But, and this is the main point, 
this decay depends non-uniformly on the  discretization, which leads to the above-mentioned uncertainty. 

2.) The second discretization is done from the right by discretizing the integral in the 
Hausdorff operator by a midpoint rule. This can be seen as using the projector $Q_N:L^2(0,1) \to L^2(0,1)$ onto
piecewise constant functions. The corresponding matrix $(T Q_N)^* (T Q_N)$ has entries which 
can be calculated using the $\text{Li}_2$ polylogarithmus function. (Note that  more advanced quadrature 
rules do not lead to simple expressions for the corresponding matrices, which justifies out choice of 
the midpoint rule). 

For both discretizations, the  eigenvalues of the  matrices are then calculated  using Matlab's eig-function.  
\begin{figure}
\includegraphics[width=0.47\textwidth]{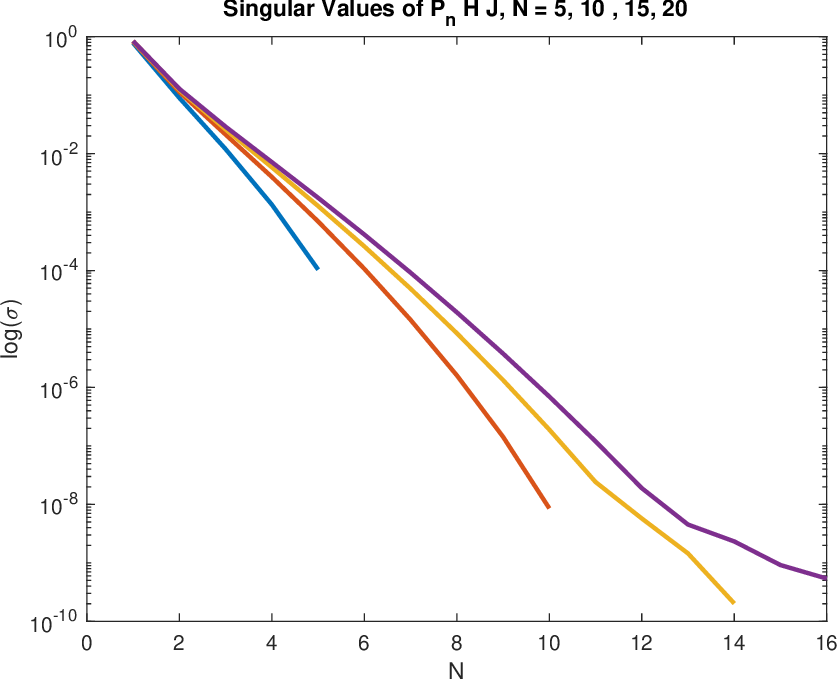}\hfill  
\includegraphics[width=0.47\textwidth]{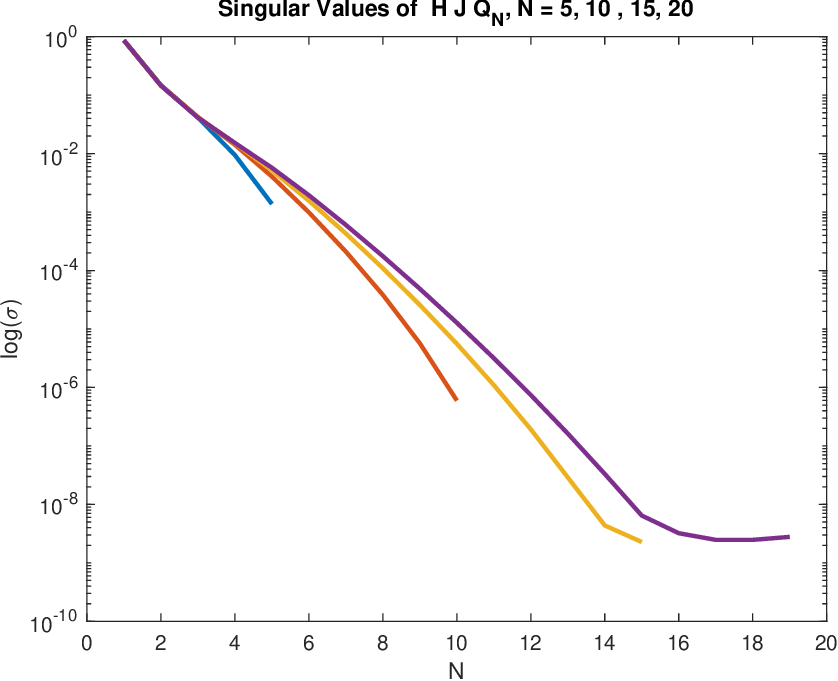} 
\\[1ex]
\includegraphics[width=0.47\textwidth]{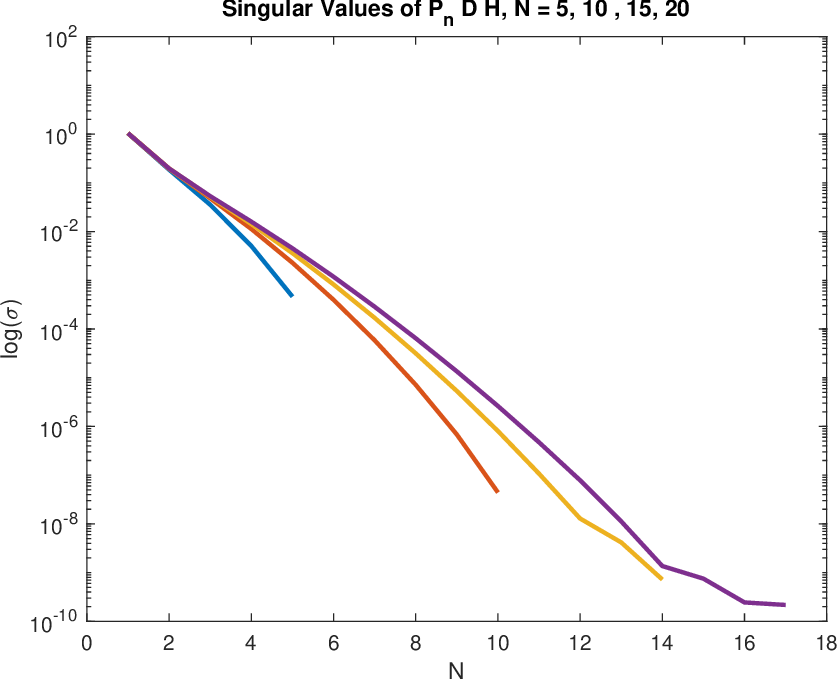}
\hfill 
\includegraphics[width=0.47\textwidth]{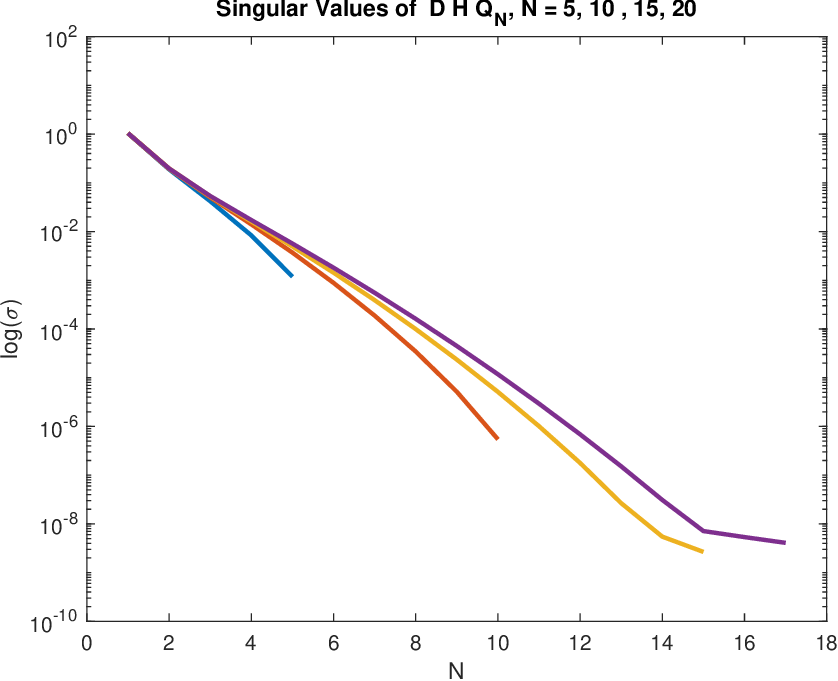}
\caption{Plot of $\log(\sigma)$ for the two discretizations of the operators $\Ha  \, J$ (top row) 
and $D \, \Ha = \Ha\, C^*$ (bottom row). Left: discretization by the left using $P_N$. Right: discretization 
from the right using $Q_N$}\label{firstfig}
\end{figure}
The results are illustrated in Figure~\ref{firstfig}. The figures provide a plot of the 
logarithm of the singular values for different discretization levels $N = 5, 10,15,20$. 
The top row corresponds to the operator $\Talt = \Ha  \, J$ 
(as considered in \cite{Gerthnum}) and the bottom row to $ T = D\, \Ha $, and the 
left column corresponds to a discretization by $P_N$ from the left and the right column to 
a discretization  by $Q_N$ from the right.

The two different discretizations give qualitatively similar results. 
Since an exponential decay in this  semi-logarithmic plot corresponds to a linear decay, 
we find that in any case the discretized operators show exponential decay. 
However, the asympotics come with a huge error margin such that 
no conclusion about the decay for the original operators can be drawn.

\begin{figure}
\includegraphics[width=0.47\textwidth]{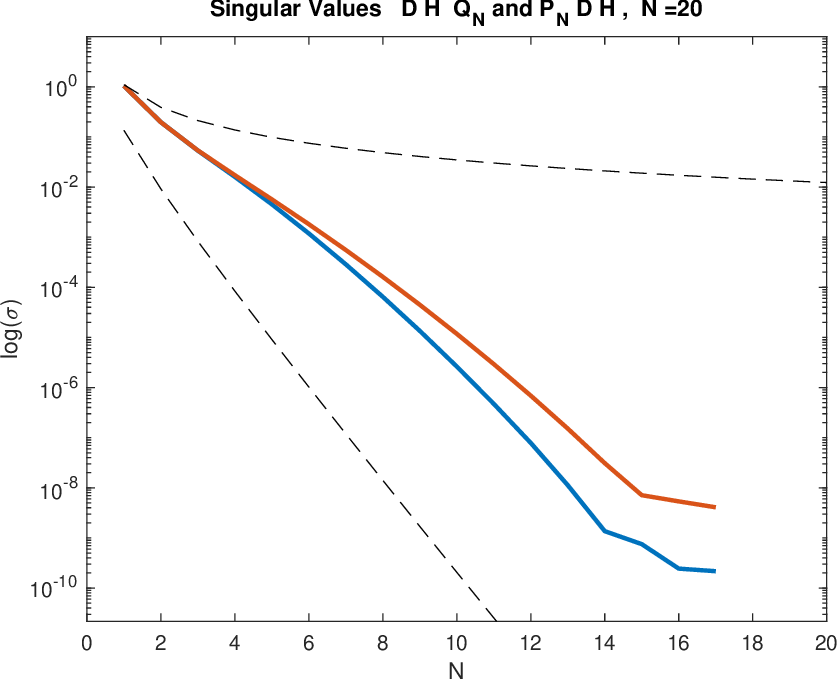}\hfill  
\includegraphics[width=0.47\textwidth]{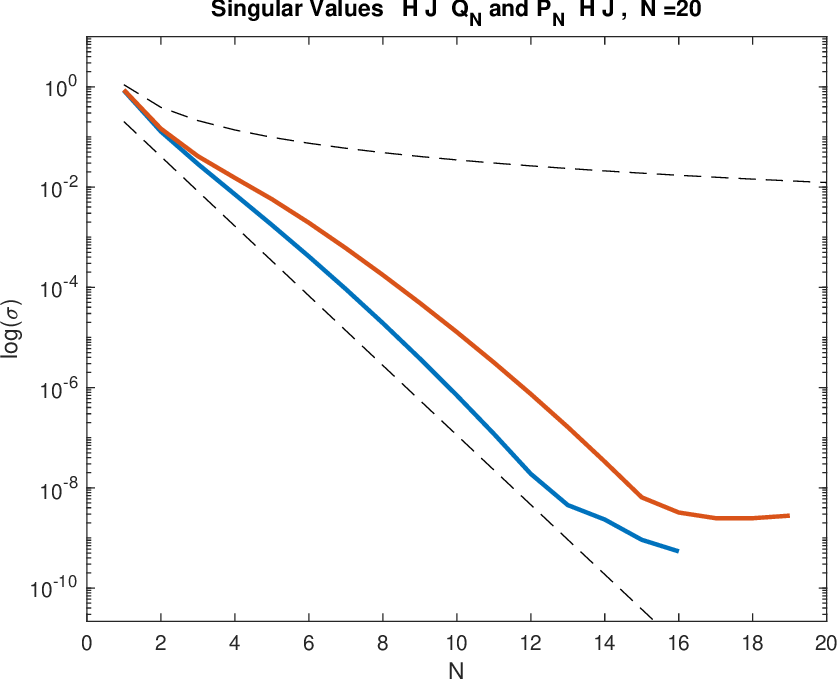} 
\caption{Plot of $\log(\sigma)$ for the two discretizations of the operators $D \, \Ha = \Ha\, C^*$ (left) and 
$\Ha  \, J$ (right) for $N = 20$. Dashed lines: Upper and lower theoretical estimates, $i^{-3/2}$ and 
$\exp(-c i)$ with $c = 1.6$ for $\Ha  \, J$ (right) and $i^{-3/2}$ and $\exp(-2 i)/i$ for $D \, \Ha $ (left).
}
\label{secondfig}
\end{figure}
The next figure, Figure~\ref{secondfig} further illustrates this issue: 
her we simultaneously plot the log of the singular values for the two discretization procedures 
for $N = 20$ (for each operator separately). 
The dashed lines correspond to the theoretical upper and lower bounds from the analysis above.

On the left we illustrate the results for $D \, \Ha $ and on the 
right for $\Ha  \, J$. Note that the lower bound  in Theorem~\ref{thm:lowerbounds}  for 
$D \, \Ha $ seem to be too pessimistic, and 
one might conjecture that a lower bound  $\exp(-c i)$ as for $\Ha  \, J$ holds.
What can be observed is that the singular values  $\sigma_i$ of the different
discretization methods 
agree only for indices $i =1\ldots,4$. In other words every singular value  
higher than the 4th one is here discretization-dependent and does not give any clue 
about the true values of the original operator. In particular, the question about 
the degree of ill-posedness (moderate or severely) 
is still open and cannot be answered by such numerical calculations.

\section{When is the composition $\Ti \, \Tii$ compact for non-compact operators  $\Ti$ and $\Tii$?}\label{sec:ncnc}
In this section, we would further investigate the above curious observation 
of a compact product of non-compact operators. 
There are some studies of this phenomenon in operator theory. 
For instance, there  is the notion of a {\em power compact operator} \cite{Barria80,Tacon79}, that is, an operator $T$ (not necessarily 
compact), where some power $T^k$ is a compact operator.  In these references, a
characterization of power-compactness was given.
In \cite[pp~579--580]{DS58}, it is remarked that 
such operators have similar spectral properties as compact operators and that weakly compact 
operators in the space of continuous functions are instances of such power compact operators. 
All these are examples of compact products  of non-compact operators.

Related to this are polynomially compact operators, where $p(T)$ is compact for some 
polynomial~\cite{Olsen71}. Also in this reference, there is even a characterization when the product of
two operators is compact; see Theorem~2.3 in \cite{Olsen71}:

\begin{thm}
Let $\Ti,\Tii$ be two bounded operators on a separable Hilbert space.  Then the product $\Ti \Tii$ is compact 
if and only if there is a projection $E$ such that $\Ti E$ and $(I-E)\Tii$  are both compact. 
\end{thm}

The question when the composition $\Ti\Tii$ is compact for non-compact $\Ti$ and $\Tii$ can also be stated in terms of abstract 
operator theory. Recall that the set of compact operators $K: \X \to \X$ constitute an ideal $I(\X)$ in 
the algebra of bounded operators $B(\X)$. Thus, one can construct the quotient algebra $B(\X)/I(\X)$, called the Calkin algebra. The case that $\Ti\Tii$ is compact (and hence is an element in $I(\X)$) for 
$\Ti,\Tii$ non-compact means in the Calkin algebra that $[\Ti][\Tii] = 0$, for $[\Ti]$, $[\Tii]$ $\not = 0$, 
where $[.]$ denotes equivalence classes. In other words,  $[\Ti]$ 
(and $[\Tii]$ as well) is a zero divisor.

The zero divisors of the Calkin algebra have been characterized by  Pfaffenberger in \cite{Pfaff70}
as the complement of the semi-Fredholm operators; cf.~\cite[Corollary 3.7]{Pfaff70}, 
where semi-Fredholm means closed range together with finite-dimensional nullspace or 
finite-dimensional codimension of the range.  Essentially this means that 
for all type I ill-posed operators, we can find a non-compact operator such that 
their product is compact (for type II operators this is trivial). 

Hence, the identity ``$\text{Non-compact} \circ \text{Non-compact} = \text{Compact}$''
is almost a characterization of type~I ill-posedness. For a better understanding of these results, 
we will detail and prove (following the lines of \cite{Pfaff70} and the references therein) 
the corresponding results below for our simplified case of Hilbert space operators.  

In the following we assume throughout that  $\X,\Y$ are Hilbert spaces.
Let us start with a simple well-known lemma:
\begin{lem}\label{lemma1}
Let $\Ti: \X \to \Y$ be  a bounded operator with non-closed range. 
Then for any $\epsilon >0$ there exists a $x \in N(\Ti)^\bot$ such that 
\[ \|\Ti x\| \leq \epsilon \|x\|.\]
\end{lem}
\begin{proof}
If the conclusion would not hold, then we had an $\epsilon$ such that 
for all $x\in N(\Ti)^\bot$, with \mbox{$\| \Ti x\| \geq \epsilon \|x\|$}. This means that 
the pseudo-inverse restricted to $\range(\Ti)$, \mbox{$\Ti^\dagger:\range(\Ti) \to N(\Ti)^\bot$} satisfies 
\[\|\Ti^\dagger y\| \leq \epsilon^{-1} \|\Ti \Ti^\dagger y\| = \epsilon^{-1} \|Q_{\range} y\| \leq \epsilon^{-1} \]
and hence is bounded, where $Q_{\range}$ denotes the orthogonal projection onto $\overline{\range(\Ti)}$ (see the 
Moore-Penrose inverse, \cite[Prop. 2.3]{EHN96}.
Thus, $\Ti^\dagger$ has a bounded extension to $\overline{\range(\Ti)}$, which contradicts the 
assumption of a non-closed range of $\Ti$ via Nashed's characterization of ill-posedness \cite{Nashed86}. 
\end{proof}

Next, we follow a construction of \cite{GMF}. 
\begin{pro}\label{pro2}
Let $\Ti:\X\to \Y$  have non-closed range. Then there exists an
infinite-dimensional closed subspace $\tilde{\X} \subset \X$ such that $\Ti$ restricted to $\tilde{\X}$ is compact:
\[ \Ti|_{\tilde{\X}} = K, \]
where $K:{\tilde{\X}} \to \Y$ is compact. 
\end{pro}
\begin{proof}
Let us define $\Z:=N(\Ti)^\bot$ 
and the operator $\tTi: \Z\to \Y$ 
such that $\tTi =  \Ti|_{N(\Ti)^\bot}$.
Since the ranges of $\Ti$ and $\tTi$ agree, $\tTi$ has non-closed range as well. 
Note that $N(\tTi) = 0$. 

We inductively construct $\tilde{\X}$ through
 an orthonormal basis $(x_n)_{n=1}^\infty$ in $\Z$
such that \[ \|\tTi x_n\| \leq \frac{1}{2^n}. \]

Let us start with $n = 1$:
Let $x_1$ be such that $\|x_1\| = 1$ and $\|\tTi x_1\| \leq \frac{1}{2}$.
By  Lemma~\ref{lemma1} such an $x_1$ exists. 
We proceed by the induction step.
Assume that we have defined an orthonormal basis
$x_{1},\ldots x_n$  in $\Z$ such that 
\[ \|\tTi x_i\| \leq \frac{1}{2^i}, \qquad i=1,\ldots n,\]
and denote by $\X_n$ the span of the $(x_i)_{i=1}^n$, by 
$P_n$ the orthogonal projector onto $\X_n$, and by 
$Q_n = I -P_n$ the orthogonal projector onto $\X_n^\bot$ with complements taken
with respect to $\Z$.  

Note that $\X_n^\bot$ is an 
infinite-dimensional space.
Because otherwise, if it were finite-dimensional with 
$\X_n^\bot = \operatorname{span}\{z_i,i=1,k\} =: \W$, 
then 
$\W + \X_n = \Z$, and thus the image of $\Z$ under $\tTi$  is 
the span of $\Ti x_i$, $\Ti z_j$  $i=1,\ldots, n$, $j=1,\ldots,k$ and hence finite-dimensional, 
 contradicting   the assumption for $\Ti$. Thus $Q_n$ is a nonzero projection onto an 
infinite-dimensional subspace.

Note that we have $\tTi = \tTi P_n + \tTi Q_n$. Since the sum of two 
closed subspaces is closed if one of them is finite-dimensional, 
we note that, as $\tTi P_n$ is finite-dimensional, 
the range of  $\tTi Q_n$  cannot be closed as otherwise the range of $\tTi$ would be closed as well. 

By  Lemma~\ref{lemma1} applied to $\tTi Q_n$, we thus find a 
normalized $y_{n+1}$ with $\|\tTi Q_n y_{n+1} \| \leq   \frac{1}{2^{n+1}} \|y_{n+1}\|$ 
and $y_{n+1}\in N(\tTi Q_n)^\bot$.
Since $N(\tTi Q_n)^\bot = \overline{R(Q_n \tTi^*)}$, it follows that $y_{n+1}  \in \X_n^\bot$ and thus 
$ \|y_{n+1}\| = \|Q_n y_{n+1}\|$. 
Take $x_{n+1} = \frac{Q_n y_{n+1}}{\|Q_n y_{n+1}\|}$.
Then $x_{1},\ldots, x_{n+1}$ is orthonormalized and 
\[ \|\tTi  x_{n+1}\| \leq  \frac{1}{2^{n+1}} \frac{\|y_{n+1}\|}{{\|Q_n y_{n+1}\|}} 
 =  \frac{1}{2^{n+1}} \frac{\|Q_n y_{n+1}\|}{{\|Q_n y_{n+1}\|}}  = \frac{1}{2^{n+1}}.
\]

Now take the operator  $K_n:= \tTi P_n$. Since $\range{(P_n)} \subset N(\Ti)^\bot$
we also have $K_n = \Ti P_n$. This operator can be written with help of the 
orthonormal basis $x_i$ as 
\[ K_n x = \sum_{i=1}^n  \Ti x_i (x,x_i)_{\X} \]
and obviously has finite-dimensional range. 
Then, 
\[ \|(K_n - K_m) x\| = \|\sum_{i=n}^m  \Ti x_i (x,x_i)_{\X}\| \leq 
 \sum_{i=n}^m \|\Ti x_i\| \|x\| \leq  
  \sum_{i=n}^m  \frac{1}{2^i}\|x\| \leq  \frac{1}{2^{n}} \|x\|.
\]
Thus $K_n$ is a Cauchy-sequence of finite-range  operators which therefore converge 
to a compact limit operator $K_n \to K$. By construction, we have 
that $K x_i = \Ti x_i$ and thus $\tilde{\X} = \operatorname{span}\{x_i\}$ and $K$ provides 
the subspace and compact operator that we were looking for. 
\end{proof} 

As mentioned above, we have adapted the proof of the proposition from  \cite[Thm.~4.1, p.72]{GMF}, where 
it was done in a more general Banach-space setting with biorthogonal sequences. The similar construction 
was used in \cite{Schech}, the result of which was employed in the proof of \cite{Pfaff70}.

We now state the first main result of this section:
\begin{thm}\label{main1}
Let $\Ti:\X \to \Y$ be a bounded operator. 

\begin{enumerate} 
\item\label{one} Assume that $\Ti$ has a non-closed range. Then there exists a non-compact operator $\Tii: \X\to \X$ such 
that $\Ti\Tii$ is compact. 

\item\label{two} Conversely, suppose that $N(\Ti)$ is finite-dimensional and assume that 
there exists a non-compact $\Tii:\X\to \X$ such that $\Ti \Tii$ is compact. 
Then $\Ti$ has a non-closed range. 
\end{enumerate}
\end{thm}
\begin{proof}
Ad \ref{one}:
By taking $\Tii$ as the projector onto the infinite-dimensional space $\tilde{\X}$ from Proposition~\ref{pro2}, 
we have that $\Ti\Tii = K$ is compact. Since $P$ projects onto an infinite-dimensional space, it cannot be
compact. 

Ad \ref{two}: Assume that $\Ti$ has a closed range. Then, by classical results in 
regularization theory~\cite{EHN96}, there exists a bounded pseudo-inverse $\Ti^\dagger$. 
The Moore-Penrose equations give $\Ti^\dagger \Ti  = Q_{N(\Ti)^\bot}$, where  
$Q_{N(\Ti)^\bot}$ is the 
orthogonal projector onto $N(\Ti)^\bot$. Let $P = I - Q_{N(\Ti)^\bot}$ be the corresponding 
projector onto $N(\Ti)$, which has thus finite-dimensional range by assumption. Then, 
since $\Ti\Tii$ is compact, 
$\Ti^\dagger \Ti \Tii = Q_{N(\Ti)^\bot} \Tii$ is a compact operator as well, and hence 
$\Tii = P \Tii + Q_{N(\Ti)^\bot} \Tii$ is the sum of a compact operator and one with 
finite-dimensional range, hence compact as well. This is a contradiction to $\Tii$ being non-compact. 
Thus  $\Ti$ cannot have a closed range. 
\end{proof}

\begin{rem}\rm 
The assumption on the finite-dimensionality of $N(\Ti)$ is necessary since otherwise $\Ti$ could indeed 
have a closed range: For instance, if $\Ti$ has closed range with infinite-dimensional nullspace, 
then we  take as $\Tii$ the projector onto $N(\Ti)$, which is non-compact. The product $\Ti\Tii = 0$ is 
thus compact. 
\end{rem} 

We can collect the findings into a theorem, which is the second main result of this section:
\begin{thm}\label{main2}
 Let $\Ti:\X \to \Y$ be a bounded operator. 
Then the following is equivalent: 
\begin{enumerate} 
\item\label{e:one} $\Ti$ has non-closed range or $N(\Ti)$ is infinite-dimensional. 
 \item\label{e:two} There exists a non-compact $\Tii:\X\to \X$ such that $\Ti\Tii$ is compact. 
 \end{enumerate}
\end{thm}

\begin{proof}
\ref{e:one} $\Rightarrow$ \ref{e:two}: In case $\Ti$ has non-closed range, we have 
shown this in Theorem~\ref{main1} and in case of an infinite-dimensional $N(\Ti)$ in 
the above remark. 

The case \ref{e:two} $\Leftarrow$ \ref{e:one} follows from Theorem~\ref{main1}, item \ref{two}.
\end{proof}

\begin{rem}\rm 
Theorem~\ref{main2} reflects essentially a result of \cite{Pfaff70}, which was shown in an even more 
general setting of Banach spaces but  stated in rather abstract   operator algebra form.
The case of $\Tii\Ti$ being compact for a given $\Ti$ and non-compact $\Tii$ can 
be easily treated by taking adjoints:  $(\Tii\Ti)^* = \Ti^* \Tii^*$. 
Since $\Tii\Ti$ is compact if and only if $(\Tii\Ti)^*$ is, we could express  the analogous results of
Theorem~\ref{main2} with $\Tii\Ti$ in place of $\Ti\Tii$ and 
$N(\Ti^*)$ in place of $N(\Ti)$.

\end{rem} 

We note that we did not require the operator $\Tii$ 
in Theorem~\ref{main2} to 
be ill-posed in the sense of Nashed. However, under 
certain injectivity conditions,  the 
property $\Ti \Tii = \text{compact}$ is {\em equivalent} to 
type~I ill-posedness of {\em both} operators, as it is the case for 
our example with the 
Ces\`{a}ro- and Hausdorff operator ($\Ti = \Ha,\Tii = C^*$):
\begin{cor}
Let $\Ti:\X\to \Y$ and $\Tii:\Y \to \Z$ be noncompact. 
Assume that $\Ti$ and $\Tii^*$ are injective, and 
assume that $\Ti \Tii$ is compact. 
Then both $\Ti$ and $\Tii$ have non-closed range, i.e., they are 
ill-posed of type~I. 
\end{cor}

\section*{Acknowledgment} 
The second named author has been supported by the German Science Foundation (DFG) under grant~HO~1454/13-1 (Project No.~453804957).

\end{document}